\documentclass[12pt]{article}

\usepackage{amssymb}
\usepackage{amsmath}
\usepackage{amsthm}
\usepackage[all]{xy}

\usepackage[colorlinks=true,citecolor=black,linkcolor=black,urlcolor=blue]{hyperref}

\newcommand{\XYMATRIX}{\xymatrix@M=6pt}
\newcommand{\aremb}{\ar@{^{(}->}}
\newcommand{\arembfrom}{\ar@{<-^{)}}}

\numberwithin{equation}{section}

\theoremstyle{plain}
\newtheorem{THM}{Theorem}

\theoremstyle{definition}
\newtheorem{DEF}[THM]{Definition}
\newtheorem{EX}[THM]{Example}

\theoremstyle{remark}

\newcommand{\quotient}[2]{\genfrac{[}{]}{0pt}{}{#1}{#2}}

\renewcommand{\le}{\leqslant}
\renewcommand{\ge}{\geqslant}

\newcommand{\0}{\varnothing}

\renewcommand{\sec}{\cap}
\renewcommand{\phi}{\varphi}
\renewcommand{\epsilon}{\varepsilon}
\newcommand{\UNION}{\bigcup}

\newcommand{\CC}{\mathbf{C}}
\newcommand{\DD}{\mathbf{D}}

\newcommand{\KK}{\mathbf{K}}

\newcommand{\union}{\cup}

\newcommand{\Boxed}[1]{\mbox{$#1$}}

\newcommand{\id}{\mathrm{id}}

\newcommand{\Ob}{\mathrm{Ob}}

\newcommand{\op}{\mathrm{op}}

\newcommand{\Arr}{\mathrm{Arr}}

\newcommand{\cod}{\mathrm{cod}}

\newcommand{\calA}{\mathcal{A}}
\newcommand{\calB}{\mathcal{B}}
\newcommand{\calC}{\mathcal{C}}
\newcommand{\calD}{\mathcal{D}}
\newcommand{\calE}{\mathcal{E}}

\newcommand{\calS}{\mathcal{S}}

\newcommand{\CHemb}{\mathbf{Ch}}

\newcommand{\CHrs}{\mathbf{Ch}_{\mathit{rs}}}

\newcommand{\PERMquo}{\mathbf{Perm}_{\mathit{quo}}}

\DeclareMathOperator{\dom}{dom}

\title{A Dual Ramsey Theorem for Permutations}
\author{Dragan Ma\v sulovi\'c\thanks{Supported by the Grant No.\ 174019 of the Ministry of Education, Science and Technological Development of the Republic of Serbia.}\\
\small Department of Mathematics and Informatics\\[-0.8ex]
\small Faculty of Sciences, University of Novi Sad\\[-0.8ex] 
\small Trg Dositeja Obradovi\'ca 3, 21000 Novi Sad, Serbia\\
\small\tt dragan.masulovic@dmi.uns.ac.rs\\
}

\date{Oct 30, 2017}

\begin{document}

\maketitle

\begin{abstract}
  In 2012 M.\ Soki\'c proved that the class of all finite permutations has the
  Ramsey property. Using different strategies
  the same result was then reproved in 2013 by J.\ B\"ottcher and J.\ Foniok, in 2014 by M.\ Bodirsky and
  in 2015 yet another proof was provided by M.\ Soki\'c.

  Using the categorical reinterpretation of the Ramsey property in this paper we prove
  that the class of all finite permutations has the dual Ramsey property as well.
  It was Leeb who pointed out in 1970 that the use of category theory can be quite helpful
  both in the formulation and in the proofs of results pertaining to structural Ramsey theory.
  In this paper we argue that this is even more the case when dealing with the dual Ramsey property.

  
\end{abstract}

\section{Introduction}

Generalizing the classical results of F.~P.~Ramsey from the late 1920's, the structural Ramsey theory originated at
the beginning of 1970's in a series of papers (see \cite{N1995} for references).
We say that a class $\KK$ of finite structures has the \emph{Ramsey property} if the following holds:
for any number $k \ge 2$ of colors and all $\calA, \calB \in \KK$ such that $\calA$ embeds into $\calB$
there is a $\calC \in \KK$
such that no matter how we color the copies of $\calA$ in $\calC$ with $k$ colors, there is a \emph{monochromatic} copy
$\calB'$ of $\calB$ in $\calC$ (that is, all the copies of $\calA$ that fall within $\calB'$ are colored by the same color).
Many classes of structures were shown to have the Ramsey property:
finite linearly ordered graphs~\cite{AH, Nesetril-Rodl},
finite posets together with an additional linear extension of the poset ordering~\cite{Nesetril-Rodl-1984},
finite linearly ordered metric spaces~\cite{Nesetril-metric}, and so on.

In 2012 M.\ Soki\'c proved that the class of all finite permutations has the Ramsey property~\cite{sokic1}.
Using different strategies the same result was then reproved in 2013 by J.\ B\"ottcher and J.\ Foniok~\cite{bottcher-foniok},
in 2014 by M.\ Bodirsky~\cite{bodirsky} and in 2015 yet another proof was provided by M.\ Soki\'c.
Discussing the Ramsey property in the context of permutations
relies on P.~J.~Cameron's reinterpretation of permutations in model-theoretic terms~\cite{cameron-perm} as follows.
From a traditional point of view
a permutation of a set $A$ is any bijection $f : A \to A$. If $A$ is finite, say $A = \{a_1, a_2, \ldots, a_n\}$,
then each permutation $f : A \to A$ can be represented as
$
    f = \begin{pmatrix}
      a_1 & a_2 & \ldots & a_n \\
      a_{i_1} & a_{i_2} & \ldots & a_{i_n}
    \end{pmatrix}
$.
So, in order to specify a permutation it suffices to specify two linear orders on $A$: the ``standard'' order
$a_1 < a_2 < \ldots < a_n$ on $A$, and the permuted order $a_{i_1} \sqsubset a_{i_2} \sqsubset \ldots \sqsubset a_{i_n}$.
In this paper we adopt P.~J.~Cameron's point of view and say that a \emph{permutation} is a triple
$(A, \Boxed<, \Boxed\sqsubset)$ where $<$ and $\sqsubset$ are linear orders on~$A$.

Using the categorical reinterpretation of the Ramsey property as proposed in~\cite{masulovic-ramsey}
we prove in this paper that the class of finite permutations has the dual Ramsey property.
Instead of embeddings, which are crucial for the notion of a subpermutation in the above ``direct''
Ramsey result, we shall consider special surjective maps that we refer to as quotient maps for permutations.
These quotient maps are strongly motivated by the notion of minors for permutations suggested, in a different context,
by E.~Lehtonen in~\cite{erkko1}.

It was Leeb who pointed out in 1970 \cite{leeb-cat} that the use of category theory can be quite helpful
both in the formulation and in the proofs of results pertaining to structural Ramsey theory.
In this paper we argue that this is even more the case when dealing with the dual Ramsey property.
Our strategy is to take a ``direct'' Ramsey result,
provide a purely categorical proof of the result and then capitalize on the Duality Principle
(an intrinsic principle of category theory) which states that if a statement is true in a category~$\CC$
then the dual of the statement is true in the opposite category~$\CC^\op$.

In Section~\ref{opos.sec.prelim} we
give a brief overview of standard notions referring to finite linearly ordered sets and category theory,
and conclude with the reinterpretation of the Ramsey property in the language of category theory.
In Section~\ref{drpperm.sec.transfer} we consider two ways to transfer the Ramsey property from a category
to another category. We first show a Ramsey-type theorem for products of categories generalizing thus
the Finite Product Ramsey Theorem for Finite Structures of M.~Soki\'c~\cite{sokic2}, and then
prove a simple result which enables us to transfer the Ramsey property from a category to its (not necessarily full) subcategory.
Using these two ``transfer principles'', starting from a categorical reinterpretation
of the Finite Dual Ramsey Theorem we infer in Section~\ref{drpperm.sec.drp}
a dual Ramsey theorem for the category of finite permutations.
  
\section{Preliminaries}
\label{opos.sec.prelim}

In this section we give a brief overview of standard notions referring to linearly ordered sets
and category theory, and conclude with the reinterpretation of the Ramsey property in the language of category theory.

\subsection{Chains and permutations}

  A \emph{chain} is a pair $(A, \Boxed<)$ where $<$ is a linear ($=$~total) order on~$A$.
  In case $A$ is finite, instead of $(A, \Boxed<)$ we shall simply write $A = \{a_1 < a_2 < \ldots < a_n\}$.

  It is easy to see that a map $f : A \to B$ between two chains $(A, \Boxed{<})$ and
  $(B, \Boxed{<})$ is an embedding if and only if $x < y \Rightarrow f(x) < f(y)$ for all
  $x, y \in A$.
  
  Let $(A, \Boxed<)$ and $(B, \Boxed\sqsubset)$ be chains such that $A \sec B = \0$.
  Then $(A \union B, \Boxed{\Boxed< \oplus \Boxed\sqsubset})$ denotes the \emph{concatenation} of
  $(A, \Boxed<)$ and $(B, \Boxed\sqsubset)$, which is a chain on $A \union B$ such that every element of $A$
  is smaller then every element of $B$, the elements in $A$ are ordered linearly by~$<$, and the elements of $B$
  are ordered linearly by~$\sqsubset$.

  A \emph{permutation} on a set $A$ is a triple $(A, \Boxed<, \Boxed\sqsubset)$ where $<$ and $\sqsubset$
  are linear orders on~$A$~\cite{cameron-perm}. Again it is easy to see that an embedding of a permutation
  $(A, \Boxed<, \Boxed\sqsubset)$ into a permutation $(B, \Boxed<, \Boxed\sqsubset)$ is a map $f : A \to B$  such that
  $x < y \Rightarrow f(x) < f(y)$, and $x \sqsubset y \Rightarrow f(x) \sqsubset f(y)$,
  for all $x, y \in A$.

\subsection{Categories and functors}
\label{drpperm.sec.catth}

In order to keep the paper self-contained, in this section we provide a brief overview of some
elementary category-theoretic notions. For a detailed account of category theory we refer
the reader to~\cite{AHS}.

In order to specify a \emph{category} $\CC$ one has to specify
a class of objects $\Ob(\CC)$, a set of morphisms $\hom_\CC(\calA, \calB)$ for all $\calA, \calB \in \Ob(\CC)$,
the identity morphism $\id_\calA$ for all $\calA \in \Ob(\CC)$, and
the composition of mor\-phi\-sms~$\cdot$~so that
$\id_\calB \cdot f = f = f \cdot \id_\calA$ for all $f \in \hom_\CC(\calA, \calB)$, and
$(f \cdot g) \cdot h = f \cdot (g \cdot h)$ whenever the compositions are defined.
A morphism $f \in \hom_\CC(\calB, \calC)$ is \emph{monic} or \emph{left cancellable} if
$f \cdot g = f \cdot h$ implies $g = h$ for all $g, h \in \hom_\CC(\calA, \calB)$ where $\calA \in \Ob(\CC)$ is arbitrary.

\begin{EX}\label{drpperm.ex.CH-def}
  Finite chains and embeddings constitute a category that we denote by~$\CHemb$.
\end{EX}

\begin{EX}\label{drpperm.ex.CHrs-def}
  Following \cite{promel-voigt-surj-sets} we say that a surjection $f : \{a_1 < a_2 < \ldots < a_n\} \to
  \{b_1 < b_2 < \ldots < b_k\}$ between two finite chains
  is \emph{rigid} if $\min f^{-1}(b_i) < \min f^{-1}(b_j)$ whenever $i < j$.
  Equivalently, $f$ is rigid if for every $s \in \{1, \ldots, n\}$ there is a $t \in \{1, \ldots, k\}$ such that
  $f(\{a_1, \ldots, a_s\}) = \{b_1, \ldots, b_t\}$. (In other words, a rigid surjection maps an initial segment of a chain
  onto an initial segment of the other chain. Other than that, a rigid surjection is not required to respect the linear orders in question.)
  
  The composition of two rigid surjections is again a rigid surjection, so finite chains and rigid surjections constitute a category
  which we denote by~$\CHrs$.
\end{EX}

For a category $\CC$, the \emph{opposite category}, denoted by $\CC^\op$, is the category whose objects
are the objects of $\CC$, morphisms are formally reversed so that
$
  \hom_{\CC^\op}(\calA, \calB) = \hom_\CC(\calB, \calA)
$,
and so is the composition:
$
  f \cdot_{\CC^\op} g = g \cdot_\CC f
$.

A category $\DD$ is a \emph{subcategory} of a category $\CC$ if $\Ob(\DD) \subseteq \Ob(\CC)$ and
$\hom_\DD(\calA, \calB) \subseteq \hom_\CC(\calA, \calB)$ for all $\calA, \calB \in \Ob(\DD)$.
A category $\DD$ is a \emph{full subcategory} of a category $\CC$ if $\Ob(\DD) \subseteq \Ob(\CC)$ and
$\hom_\DD(\calA, \calB) = \hom_\CC(\calA, \calB)$ for all $\calA, \calB \in \Ob(\DD)$.

A \emph{functor} $F : \CC \to \DD$ from a category $\CC$ to a category $\DD$ maps $\Ob(\CC)$ to
$\Ob(\DD)$ and maps morphisms of $\CC$ to morphisms of $\DD$ so that
$F(f) \in \hom_\DD(F(\calA), F(\calB))$ whenever $f \in \hom_\CC(\calA, \calB)$, $F(f \cdot g) = F(f) \cdot F(g)$ whenever
$f \cdot g$ is defined, and $F(\id_\calA) = \id_{F(\calA)}$.

Categories $\CC$ and $\DD$ are \emph{isomorphic} if there exist functors $F : \CC \to \DD$ and $G : \DD \to \CC$ which are
inverses of one another both on objects and on morphisms.

The \emph{product} of categories $\CC_1$ and $\CC_2$ is the category $\CC_1 \times \CC_2$ whose objects are pairs $(\calA_1, \calA_2)$
where $\calA_1 \in \Ob(\CC_1)$ and $\calA_2 \in \Ob(\CC_2)$, morphisms are pairs $(f_1, f_2) : (\calA_1, \calA_2) \to (\calB_1, \calB_2)$ where
$f_1 : \calA_1 \to \calB_1$ is a morphism in $\CC_1$ and $f_2 : \calA_2 \to \calB_2$ is a morphism in $\CC_2$.
The composition of morphisms is carried out componentwise: $(f_1, f_2) \cdot (g_1, g_2) = (f_1 \cdot g_1, f_2 \cdot g_2)$.
Clearly, if $\tilde \calA = (\calA_1, \calA_2)$ and $\tilde \calB = (\calB_1, \calB_2)$ are objects of $\CC_1 \times \CC_2$
then
$$
  \hom_{\CC_1 \times \CC_2}(\tilde \calA, \tilde \calB) = \hom_{\CC_1}(\calA_1, \calB_1) \times \hom_{\CC_2}(\calA_2, \calB_2).
$$

An \emph{oriented multigraph} $\Delta$ consists of a collection (possibly a class) of vertices $\Ob(\Delta)$,
a collection of arrows $\Arr(\Delta)$, and two maps $\dom, \cod : \Arr(\Delta) \to \Ob(\Delta)$ which
assign to each arrow $f \in \Arr(\Delta)$ its domain $\dom(f)$ and its codomain $\cod(f)$.
If $\dom(f) = \gamma$ and $\cod(f) = \delta$ we write briefly $f : \gamma \to \delta$.
Intuitively, an oriented multigraph is a ``category without composition''. Therefore,
each category $\CC$ can be understood as an oriented multigraph
whose vertices are the objects of the category and whose arrows are the morphisms of the category.
A \emph{multigraph homomorphism} between oriented multigraphs $\Gamma$ and $\Delta$
is a pair of maps (which we denote by the same symbol) $F : \Ob(\Gamma) \to \Ob(\Delta)$ and
$F : \Arr(\Gamma) \to \Arr(\Delta)$ such that if $f : \sigma \to \tau$ in $\Gamma$, then
$F(f) : F(\sigma) \to F(\tau)$ in $\Delta$.

Let $\CC$ be a category. For any oriented multigraph $\Delta$, a \emph{diagram in $\CC$ of shape $\Delta$}
is a multigraph homomorphism $F : \Delta \to \CC$. Intuitively, a diagram in $\CC$ is an
arrangement of objects and morphisms in $\CC$ that has the shape of~$\Delta$.
A diagram $F : \Delta \to \CC$ is \emph{commutative} if morphisms along every two paths between the same
nodes compose to give the same morphism.

A diagram $F : \Delta \to \CC$ is \emph{has a commutative cocone in $\CC$} if there exists a $\calC \in \Ob(\CC)$
and a family of morphisms $(e_\delta : F(\delta) \to \calC)_{\delta \in \Ob(\Delta)}$ such that for every
arrow $g : \delta \to \gamma$ in $\Arr(\Delta)$ we have $e_\gamma \cdot F(g) = e_\delta$:
$$
  \xymatrix{
     & \calC & \\
    F(\delta) \ar[ur]^{e_\delta} \ar[rr]_{F(g)} & & F(\gamma) \ar[ul]_{e_\gamma}
  }
$$
(see Fig.~\ref{nrt.fig.3} for an illustration).
We say that $\calC$ together with the family of morphisms
$(e_\delta)_{\delta \in \Ob(\Delta)}$ is a \emph{commutative cocone in $\CC$ over the diagram~$F$}.

\begin{figure}
  $$
  \xymatrix{
    & & & & & \exists \calC &
  \\
    \bullet & \bullet & \bullet
    & & \calB_1 \ar@{.>}[ur] & \calB_2 \ar@{.>}[u] & \calB_1 \ar@{.>}[ul]
  \\
    \bullet \ar[u] \ar[ur] & \bullet \ar[ur] \ar[ul] & \bullet \ar[ul] \ar[u]
    & & \calA_1 \ar[u]^{f_1} \ar[ur]_(0.3){f_2} & \calA_2 \ar[ur]^(0.3){f_4} \ar[ul]_(0.3){f_3} & \calA_2 \ar[ul]^(0.3){f_5} \ar[u]_{f_6}
  \\
    & \Delta \ar[rrrr]^F  & & & & \CC  
  }
  $$
  \caption{A diagram in $\CC$ (of shape $\Delta$) with a commutative cocone}
  \label{nrt.fig.3}
\end{figure}
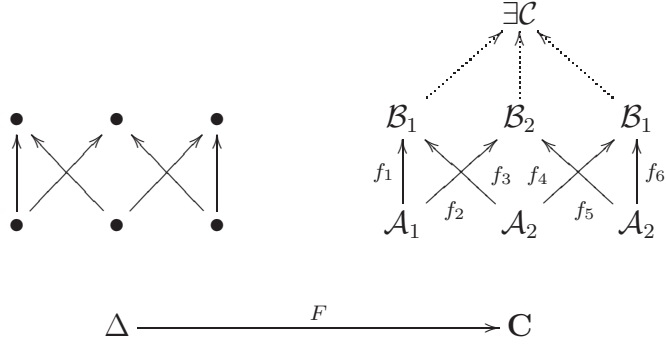

\subsection{The Ramsey property in the language of category theory}

Let $\CC$ be a category and $\calS$ a set. We say that
$
  \calS = \Sigma_1 \union \ldots \union \Sigma_k
$
is a \emph{$k$-coloring} of $\calS$ if $\Sigma_i \sec \Sigma_j = \0$ whenever $i \ne j$.
For an integer $k \ge 2$ and $\calA, \calB, \calC \in \Ob(\CC)$ we write
$
  \calC \longrightarrow (\calB)^{\calA}_k
$
to denote that for every $k$-coloring
$
  \hom_\CC(\calA, \calC) = \Sigma_1 \union \ldots \union \Sigma_k
$
there is an $i \in \{1, \ldots, k\}$ and a morphism $w \in \hom_\CC(\calB, \calC)$ such that
$w \cdot \hom_\CC(\calA, \calB) \subseteq \Sigma_i$.

\begin{DEF}
  A category $\CC$ has the \emph{Ramsey property} if
  for every integer $k \ge 2$ and all $\calA, \calB \in \Ob(\CC)$
  there is a $\calC \in \Ob(\CC)$ such that $\calC \longrightarrow (\calB)^{\calA}_k$.
  A category $\CC$ has the \emph{dual Ramsey property} if $\CC^\op$ has the Ramsey property.
\end{DEF}

Clearly, if $\CC$ and $\DD$ are isomorphic categories, then one of them has the (dual) Ramsey property if and only if
the other one does.

\begin{EX}\label{cerp.ex.FRP-ch}
    The category $\CHemb$ of finite chains and embeddings (Example~\ref{drpperm.ex.CH-def}) has the Ramsey property.
    This is just a reformulation of the Finite Ramsey Theorem~\cite{Ramsey}:

  \begin{quote}
    For all positive integers $k$, $a$, $m$ there is a positive integer $n$ such that
    for every $n$-element set $C$ and every $k$-coloring of the set $\binom Ca$ of all $a$-element subsets of $C$
    there is an $m$-element subset $B$ of $C$ such that $\binom Ba$ is monochromatic.
  \end{quote}
\end{EX}

\begin{EX}\label{cerp.ex.FDRT-ch}
    The category $\CHrs$ of finite chains and rigid surjections (Example~\ref{drpperm.ex.CHrs-def})
    has the dual Ramsey property. This is just a reformulation of the Finite Dual Ramsey Theorem \cite{GR}:

  \begin{quote}
    For all positive integers $k$, $a$, $m$ there is a positive integer $n$ such that
    for every $n$-element set $C$ and every $k$-coloring of the set $\quotient Ca$ of all partitions of
    $C$ with exactly $a$ blocks there is a partition $\beta$ of $C$ with exactly $m$ blocks such that
    the set of all partitions from $\quotient Ca$ which are coarser than $\beta$ is monochromatic.
  \end{quote}
  
  \noindent
  Namely, it was observed in~\cite{promel-voigt-surj-sets} that
  each partition of a finite linearly ordered set can be uniquely represented by the
  rigid surjection which takes each element of the underlying set to the minimum of the block it belongs to.
\end{EX}

\section{Transferring the Ramsey property between categories}
\label{drpperm.sec.transfer}

A typical generalization of the Finite Ramsey Theorem is the Finite Product Ramsey Theorem (Theorem~5 in~\cite[Ch.~5.1]{GRS})
which provides a Ramsey-type result for finite tuples of finite sets. This classical theorem was generalized by M.~Soki\'c
to Finite Product Ramsey Theorem for Finite Structures in~\cite{sokic2} where instead of finite tuples of finite sets we
deal with finite tuples of finite structures. We shall now provide a further generalization of this result in the form of
a Ramsey theorem for products of categories. The benefit of such a general result is that we can then invoke
Duality Principle of category theory to automatically infer statements about the dual Ramsey property.

\begin{THM}\label{sokic-prod}
  Let $\CC_1$ and $\CC_2$ be categories such that $\hom_{\CC_i}(\calA, \calB)$ is finite
  for each $i \in \{1, 2\}$ and all $\calA, \calB \in \Ob(\CC_i)$.

  $(a)$ If $\CC_1$ and $\CC_2$ both have the Ramsey property then
  $\CC_1 \times \CC_2$ has the Ramsey property.

  $(b)$ If $\CC_1$ and $\CC_2$ both have the dual Ramsey property then
  $\CC_1 \times \CC_2$ has the dual Ramsey property.
\end{THM}
\begin{proof}
  $(a)$
  Take any $k \ge 2$ and $\tilde\calA = (\calA_1, \calA_2)$, $\tilde\calB = (\calB_1, \calB_2)$ in $\Ob(\CC_1 \times \CC_2)$
  such that $\tilde\calA \to \tilde\calB$ and let us show that there is a $\tilde\calC \in \Ob(\CC_1 \times \CC_2)$
  such that $\tilde\calC \longrightarrow (\tilde\calB)^{\tilde\calA}_k$. Take $\calC_1 \in \Ob(\CC_1)$ and
  $\calC_2 \in \Ob(\CC_2)$ so that $\calC_1 \longrightarrow (\calB_1)^{\calA_1}_{k}$ and
  $\calC_2 \longrightarrow (\calB_2)^{\calA_2}_{k^t}$,
  where $t$ is the cardinality of $\hom_{\CC_1}({\calA_1},{\calC_1})$. Put $\tilde \calC = (\calC_1, \calC_2)$.

  To show that $\tilde\calC \longrightarrow (\tilde\calB)^{\tilde\calA}_k$ take any coloring
  $$
    \chi : \hom_{\CC_1 \times \CC_2}({\tilde\calA},{\tilde\calC}) \to \{1, \ldots, k\}.
  $$
  Since $\hom_{\CC_1 \times \CC_2}({\tilde\calA},{\tilde\calC}) = \hom_{\CC_1}({\calA_1},{\calC_1}) \times \hom_{\CC_2}({\calA_2},{\calC_2})$,
  the coloring $\chi$ uniquely induces the $k^t$-coloring
  $$
    \chi' : \hom_{\CC_2}({\calA_2},{\calC_2}) \to \{1, \ldots, k\}^{\hom_{\CC_1}({\calA_1},{\calC_1})}.
  $$
  By construction, $\calC_{2} \longrightarrow (\calB_{2})^{\calA_{2}}_{k^t}$,
  so there is a $w_2 : \calB_2 \to \calC_2$ such that $w_2 \cdot \hom_{\CC_2}({\calA_2},{\calB_2})$ is $\chi'$-monochromatic.
  Let
  $$
    \chi'' : \hom_{\CC_1}({\calA_1},{\calC_1}) \to \{1, \ldots, k\}
  $$
  be the $k$-coloring defined by
  $$
    \chi''(e_1) = \chi(e_1, e)
  $$
  for some $e \in w_2 \cdot \hom_{\CC_2}({\calA_2},{\calB_2})$. Note that $\chi''$ is
  independent of the choice of $e$ because $w_2 \cdot \hom_{\CC_2}({\calA_2},{\calB_2})$ is $\chi'$-monochromatic.
  Since $\CC_1$ has the Ramsey property there is a morphism
  $
    w_1 : \calB_1 \to \calC_1
  $
  such that $w_1 \cdot \hom_{\CC_1}({\calA_1},{\calB_1})$
  is $\chi''$-monochromatic. It is now easy to show that for $\tilde w = (w_1, w_2)$ we have that
  $\tilde w \cdot \hom_{\CC_1 \times \CC_2}({(\calA_1, \calA_2)},{(\calB_1, \calB_2)})$ is $\chi$-monochromatic.
  
  $(b)$
  Assume now that both $\CC_1$ and $\CC_2$ have the dual Ramsey property. Then $\CC_1^\op$ and $\CC_2^\op$ have the
  Ramsey property, whence $\CC_1^\op \times \CC_2^\op$ has the Ramsey property by~$(a)$. By definition,
  the category $(\CC_1^\op \times \CC_2^\op)^\op = \CC_1 \times \CC_2$ then has the dual Ramsey property.
\end{proof}

Another way of transferring the Ramsey property is from a category to its subcategory. (For many deep results
obtained in this fashion see~\cite{Nesetril-AllThose}.) We shall now present a simple result which
enables us to transfer the Ramsey property from a category to its (not necessarily full) subcategory.

Consider a finite, acyclic, bipartite digraph where all the arrows go from one class of vertices into the other
and the out-degree of all the vertices in the first class is~2:
$$
  
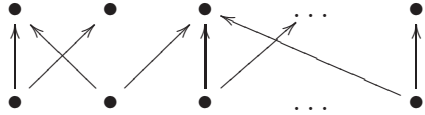
\xymatrix{
    \bullet & \bullet & \bullet & \ldots & \bullet \\
    \bullet \ar[u] \ar[ur] & \bullet \ar[ur] \ar[ul] & \bullet \ar[u] \ar[ur] & \ldots & \bullet \ar[u] \ar[ull]
  }
$$
\noindent
Such a digraph will be referred to as a \emph{binary digraph}.
A \emph{binary diagram} in a category $\CC$ is a diagram $F : \Delta \to \CC$ where $\Delta$ is a binary digraph,
$F$ takes the bottom row of $\Delta$ onto the same object, and takes the top row of $\Delta$ onto
the same object, Fig.~\ref{nrt.fig.2}.
A subcategory $\DD$ of a category $\CC$ is \emph{closed for binary diagrams} if every binary diagram
$F : \Delta \to \DD$ which has a commuting cocone in $\CC$ has a commuting cocone in~$\DD$.

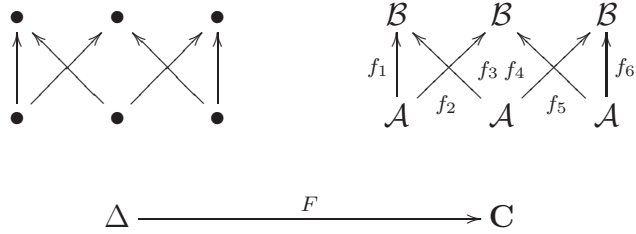
\begin{figure}
  $$
  \xymatrix{
    \bullet & \bullet & \bullet
    & & \calB & \calB & \calB
  \\
    \bullet \ar[u] \ar[ur] & \bullet \ar[ur] \ar[ul] & \bullet \ar[ul] \ar[u]
    & & \calA \ar[u]^{f_1} \ar[ur]_(0.3){f_2} & \calA \ar[ur]^(0.3){f_4} \ar[ul]_(0.3){f_3} & \calA \ar[ul]^(0.3){f_5} \ar[u]_{f_6}
  \\
    & \Delta \ar[rrrr]^F  & & & & \CC  
  }
  $$
  \caption{A binary diagram in $\CC$ (of shape $\Delta$)}
  \label{nrt.fig.2}
\end{figure}

\begin{THM}\label{nrt.thm.1}
  Let $\CC$ be a category such that every morphism in $\CC$ is monic and
  such that $\hom_\CC(\calA, \calB)$ is finite for all $\calA, \calB \in \Ob(\CC)$, and let $\DD$ be a
  (not necessarily full) subcategory of~$\CC$. If $\CC$ has the Ramsey property and $\DD$ is closed for binary diagrams,
  then $\DD$ has the Ramsey property.
\end{THM}
\begin{proof}
  Take any $k \ge 2$ and $\calA, \calB \in \Ob(\DD)$ such that $\hom_\DD(\calA, \calB) \ne \0$. Since $\DD$ is a subcategory of $\CC$
  and $\CC$ has the Ramsey property, there is a $\calC \in \Ob(\CC)$ such that $\calC \longrightarrow (\calB)^\calA_k$.
  
  Let $\hom_\CC(\calB, \calC) = \{e_1, e_2, \ldots, e_n\}$.
  Let us now construct a binary diagram in $\DD$ as follows.
  Intuitively, for each $e_i \in \hom_\CC(\calB, \calC)$ we add a copy of $\calB$ to the diagram, and whenever $e_i \cdot u = e_j \cdot v$
  for some $u, v \in \hom_\DD(\calA, \calB)$ we add a copy of $\calA$ to the diagram together with two arrows:
  one going into the $i$th copy of $\calB$ labelled by $u$ and another one going into the $j$th copy of $\calB$ labelled by~$v$
  (note that, by the construction, this diagram has a commuting cocone in $\CC$):
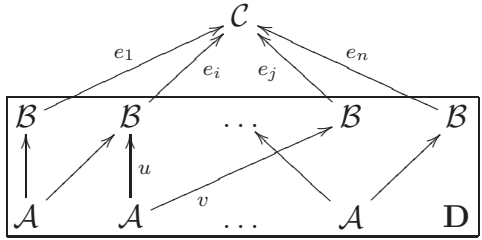
  $$
  \xymatrix{
    & & \calC
  \\
    \calB \ar[urr]^{e_1} & \calB \ar[ur]_(0.6){e_i} & \ldots & \calB \ar[ul]^(0.6){e_j} & \calB \ar[ull]_{e_n}
  \\
    \calA \ar[u] \ar[ur] & \calA \ar[urr]_(0.3){v} \ar[u]_{u} & \ldots & \calA \ar[ur] \ar[ul] & \DD
    \save "2,1"."3,5"*[F]\frm{} \restore
  }
  $$
  Formally, let $\Delta$ be the binary diagram whose objects are
  \begin{align*}
    \Ob(\Delta) = \{1, 2, \ldots, n\} \union \{(u, v, i, j) : \; & 1 \le i,j \le n; \; i \ne j;\\
                  &u, v \in \hom_\DD(\calA, \calB); \; e_i \cdot u = e_j \cdot v\}
  \end{align*}
  and whose arrows are of the form $u : (u, v, i, j) \to i$ and $v : (u, v, i, j) \to j$.
  Let $F : \Delta \to \DD$ be the following diagram whose action on objects is:
  \begin{align*}
    F(i) &= \calB, && 1 \le i \le n,\\
    F((u, v, i, j)) &= \calA, && e_i \cdot u = e_j \cdot v,
  \end{align*}
  and whose action on morphisms is $F(g) = g$:
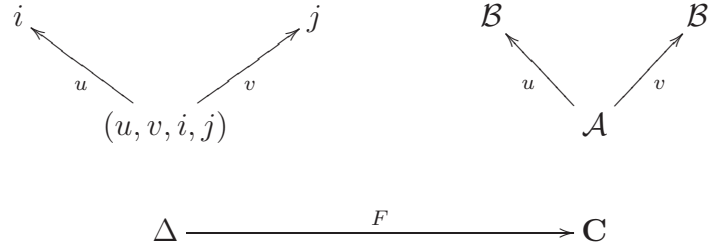
  $$
  \xymatrix{
    i &  & j
    & & \calB &  & \calB
  \\
      & (u, v, i, j) \ar[ur]_v \ar[ul]^u &  
    & &   & \calA \ar[ur]_v \ar[ul]^u & 
  \\
    & \Delta \ar[rrrr]^F  & & & & \CC  
  }
  $$

  As we have already observed in the informal discussion above, the diagram $F : \Delta \to \DD$ has a commuting cocone in $\CC$,
  so, by the assumption, it has a commuting cocone in~$\DD$. Therefore, there is
  a $\calD \in \Ob(\DD)$ and morphisms $f_i : \calB \to \calD$, $1 \le i \le n$, such that the following diagram in $\DD$ commutes:
  $$
  
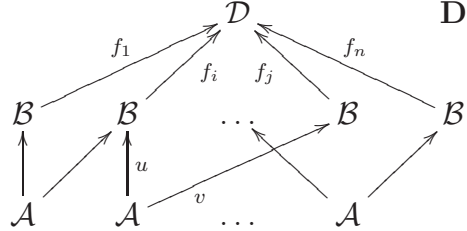
\xymatrix{
    & & \calD & & \DD
  \\
    \calB \ar[urr]^{f_1} & \calB \ar[ur]_(0.6){f_i} & \ldots & \calB \ar[ul]^(0.6){f_j} & \calB \ar[ull]_{f_n}
  \\
    \calA \ar[u] \ar[ur] & \calA \ar[urr]_(0.3){v} \ar[u]_{u} & \ldots & \calA \ar[ur] \ar[ul] & 
  }
  $$
  Let us show that in $\DD$ we have $\calD \longrightarrow (\calB)^\calA_k$. Take any $k$-coloring
  $$
    \hom_\DD(\calA, \calD) = \Sigma_1 \union \ldots \union \Sigma_k,
  $$
  and define a $k$-coloring
  $$
    \hom_\CC(\calA, \calC) = \Sigma'_1 \union \ldots \union \Sigma'_k
  $$
  as follows. For $j \in \{2, \ldots, k\}$ let
  $$
    \Sigma'_j = \{e_s \cdot u : 1 \le s \le n, u \in \hom_\DD(\calA, \calB), f_s \cdot u \in \Sigma_j \},
  $$
  and then let
  $$
    \Sigma'_1 = \hom_\CC(\calA, \calC) \setminus \UNION_{j=2}^k \Sigma'_j.
  $$
  Let us show that $\Sigma'_i \sec \Sigma'_j = \0$ whenever $i \ne j$. By the definition of $\Sigma'_1$
  it suffices to consider the case where $i \ge 2$ and $j \ge 2$. 
  Assume, to the contrary, that there is an $h \in \Sigma'_i \sec \Sigma'_j$ for some $i \ne j$, $i \ge 2$, $j \ge 2$.
  Then $h = e_s \cdot u$ for some $s$ and some $u \in \hom_\DD(\calA, \calB)$ such that $f_s \cdot u \in \Sigma_i$ and
  $h = e_t \cdot v$ for some $t$ and some $v \in \hom_\DD(\calA, \calB)$ such that $f_t \cdot v \in \Sigma_j$.
  Then $e_s \cdot u = h = e_t \cdot v$. Clearly, $s \ne t$ and we have that $(u, v, s, t) \in \Ob(\Delta)$.
  (Suppose, to the contrary, that $s = t$. Then
  $e_s \cdot u = e_t \cdot v$ implies $u = v$ because $e_s = e_t$ and all the morphisms in $\CC$ are monic.
  But then $\Sigma_i \ni f_s \cdot u = f_t \cdot v \in \Sigma_j$, which contradicts the assumption that
  and $\Sigma_i \sec \Sigma_j = \0$.) Consequently,
  $f_s \cdot u = f_t \cdot v$ because $\calD$ and morphisms $f_i : \calB \to \calD$, $1 \le i \le n$, form
  a commuting cocone over $F : \Delta \to \DD$ in~$\DD$. Therefore, $f_s \cdot u = f_t \cdot v \in \Sigma_i \sec \Sigma_j$,
  which is not possible.
  
  Since, by construction, $\calC \longrightarrow (\calB)^\calA_k$, there is an $e_\ell \in \hom_\CC(\calB, \calC)$ and a $j$ such that
  $e_\ell \cdot \hom_\CC(\calA, \calB) \subseteq \Sigma'_j$. Let us show that
  $$
    f_\ell \cdot \hom_\DD(\calA, \calB) \subseteq \Sigma_j.
  $$
  Assume, first, that $j \ge 2$ and take any $u \in \hom_\DD(\calA, \calB)$. Since $e_\ell \cdot u \in \Sigma'_j$ it follows by the definition
  of $\Sigma'_j$ that $f_\ell \cdot u \in \Sigma_j$.
  
  Assume now that $j = 1$ and take any $u \in \hom_\DD(\calA, \calB)$. Suppose that $f_\ell \cdot u \notin \Sigma_1$.
  Then $f_\ell \cdot u \in \Sigma_m$ for some $m \ge 2$. But then $e_\ell \cdot u \in \Sigma'_m$.
  On the other hand, $e_\ell \cdot u \in \Sigma'_1$ by assumption ($j = 1$).
  This is in contradiction with the construction of $\Sigma'_1$.
\end{proof}

\section{A dual Ramsey theorem for permutations}
\label{drpperm.sec.drp}

Every dual Ramsey theorem relies on some notion of a ``surjective structure map''.
In case of permutations an appropriate notion has been suggested, in a different context,
by E.~Lehtonen in~\cite[Section 6]{erkko1} as follows. Let
$\sigma = a_{i_1} \sqsubset a_{i_2} \sqsubset \ldots \sqsubset a_{i_n}$ be a permutation of a finite linearly ordered set
$A = \{a_1 < a_2 < \ldots < a_n \}$ and let $\Pi$ be a partition of~$A$. Define $f_\Pi : A \to A$ by $f_\Pi(x) = \min_<([x]_\Pi)$, where
$[x]_\Pi$ denotes the block of $\Pi$ that contains $x$ and $\min_<$ means that the minimum is taken with respect to~$<$
(the ``standard'' ordering of~$A$).
Then take the tuple $f_\Pi(\sigma) = (f_\Pi(a_{i_1}), f_\Pi(a_{i_2}), \ldots, f_\Pi(a_{i_n}))$ and
remove the repeated elements leaving only the first occurrence of each. What remains is a permutation of $f_\Pi(A)$ that we refer to
as the \emph{quotient of~$\sigma$ by~$\Pi$}.

\begin{EX}\label{drpperm.ex.erkkosminors}
  Let $A = \{ 0 < 1 < 2 < 3 < 4 < 5 < 6 < 7 < 8 < 9\}$
  be a finite linearly ordered set and let
  $$
    \sigma = 6 \sqsubset 7 \sqsubset 9 \sqsubset 3 
    \sqsubset 2 \sqsubset 8 \sqsubset 4 \sqsubset 1 \sqsubset 0 \sqsubset 5 
  $$
  be a permutation of~$A$. Let $\Pi$ be the following partition of~$A$:
  $$
    \Pi = \big\{
      \{0, 1, 4, 9\},
      \{2, 6, 8\},
      \{3, 7\},
      \{5\}
    \big\}.
  $$
  Then
  $$
    f_\Pi = \begin{pmatrix}
      0 & 1 & 2 & 3 & 4 & 5 & 6 & 7 & 8 & 9 \\
      0 & 0 & 2 & 3 & 0 & 5 & 2 & 3 & 2 & 0
    \end{pmatrix},
  $$
  so
  $$
    f_\Pi(\sigma) = (2, 3, 0, 3, 2, 2, 0, 0, 0, 5).
  $$
  Finally the quotient of~$\sigma$ by~$\Pi$ is the permutation
  $2 \prec 3 \prec 0 \prec 5$ of the ordered set
  $f_\Pi(A) = \{ 0 \mathrel{<} 2 \mathrel{<} 3 \mathrel{<} 5 \}$.
\end{EX}

\begin{DEF}
  (cf.~\cite[Section 6]{erkko1})
  Let $(A, \Boxed<, \Boxed\sqsubset)$ and $(B, \Boxed<, \Boxed\sqsubset)$ be finite permutations. We say that a surjective
  map $f : A \to B$ is a \emph{quotient map for permutations} if $f$ is a rigid surjection from $(A, \Boxed<)$ to $(B, \Boxed<)$
  as well as a rigid surjection from $(A, \Boxed\sqsubset)$ to $(B, \Boxed\sqsubset)$.
  Let $\PERMquo$ denote the category whose objects are finite permutations and whose morphisms are
  quotient maps for permutations.
\end{DEF}

\begin{EX}
  Let us write the surjective map $f_\Pi : A \to f_\Pi(A)$ from Example~\ref{drpperm.ex.erkkosminors}
  first as the rigid surjection $f_\Pi : (A, \Boxed<) \to (f_\Pi(A), \Boxed{<})$
  and then as the rigid surjection $f_\Pi : (A, \Boxed\sqsubset) \to (f_\Pi(A), \Boxed\prec)$:
  \begin{align*}
    f_\Pi &= \begin{pmatrix}
      0 & 1 & 2 & 3 & 4 & 5 & 6 & 7 & 8 & 9 \\
      0 & 0 & 2 & 3 & 0 & 5 & 2 & 3 & 2 & 0
    \end{pmatrix},&& \text{rigid w.r.t.} \binom{\Boxed<}{\Boxed<},\\
    &= \begin{pmatrix}
      6 & 7 & 9 & 3 & 2 & 8 & 4 & 1 & 0 & 5 \\
      2 & 3 & 0 & 3 & 2 & 2 & 0 & 0 & 0 & 5 
    \end{pmatrix}, && \text{rigid w.r.t.} \binom{\Boxed\sqsubset}{\Boxed\prec}.
  \end{align*}
\end{EX}

\begin{THM}[The Dual Ramsey Theorem for Permutations]\ \newline
  The category $\PERMquo$ has the dual Ramsey property.
\end{THM}
\begin{proof}
  The category $\CC = \CHrs^\op \times \CHrs^\op$ has the Ramsey property (Example~\ref{cerp.ex.FDRT-ch} and
  Theorem~\ref{sokic-prod}). Let $\DD$ be the following subcategory of $\CC$: objects of $\DD$ are pairs of ordered sets
  $((A, \Boxed<), (A, \Boxed\sqsubset))$ over the same finite set, and morphisms of $\DD$ are pairs
  $(f, f) : ((A, \Boxed<), (A, \Boxed\sqsubset)) \to ((B, \Boxed<), (B, \Boxed\sqsubset))$ such that both
  $f : (A, \Boxed<) \to (B, \Boxed<)$ and $f : (A, \Boxed\sqsubset) \to (B, \Boxed\sqsubset)$ are rigid
  surjections. It is easy to see that the categories $\DD$ and $\PERMquo$ are isomorphic, so,
  following Theorem~\ref{nrt.thm.1}, it suffices to show that $\DD$ is a subcategory of $\CC$ closed for binary diagrams.
  
  Take any $\calA = ((A, \Boxed<), (A, \Boxed\sqsubset))$ and $\calB = ((B, \Boxed<), (B, \Boxed\sqsubset))$ in $\Ob(\DD)$
  and let $F : \Delta \to \DD$ be a binary diagram which has a commuting cocone in $\CC$. Let
  $((C, \Boxed{<^C}), (D, \Boxed{\sqsubset^D}))$ together with the morphisms $e_i = (f_i, g_i)$, $1 \le i \le k$, be a
  commuting cocone in $\CC$ over~$F$:
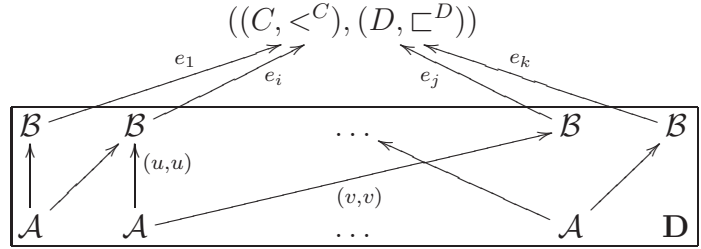
  $$
  \xymatrix{
    & & ((C, \Boxed{<^C}), (D, \Boxed{\sqsubset^D}))
  \\
    \calB \ar[urr]^{e_1} & \calB \ar[ur]_(0.6){e_i} & \ldots & \calB \ar[ul]^(0.6){e_j} & \calB \ar[ull]_{e_k}
  \\
    \calA \ar[u] \ar[ur] & \calA \ar[urr]_(0.5){(v,v)} \ar[u]_(0.65){(u,u)} & \ldots & \calA \ar[ur] \ar[ul] & \DD
    \save "2,1"."3,5"*[F]\frm{} \restore
  }
  $$
  Without loss of generality we may assume that $C \sec D = \0$.
  Recall that $f_i : (C, \Boxed{<^C}) \to (B, \Boxed<)$ and $g_i : (D, \Boxed{\sqsubset^D}) \to (B, \Boxed\sqsubset)$
  are rigid surjections; the arrows in the diagram point in the opposite direction because $\CC = \CHrs^\op \times \CHrs^\op$.

  Let $\calE = ((C \union D, \Boxed{\Boxed{<^C} \oplus \Boxed{\sqsubset^D}}),
                (C \union D, \Boxed{\Boxed{\sqsubset^D} \oplus \Boxed{<^C}}))$ and for each $i \in \{1, \ldots, k\}$ define
  $\phi_i : C \union D \to B$ as follows
  $$
    \phi_i(x) = \begin{cases}
      f_i(x), & x \in C,\\
      g_i(x), & x \in D.
    \end{cases}
  $$
  Since $f_i : (C, \Boxed{<^C}) \to (B, \Boxed<)$ is a rigid surjection, it easily follows that
  $\phi_i : (C \union D, \Boxed{\Boxed{<^C} \oplus \Boxed{\sqsubset^D}}) \to (B, \Boxed<)$ is a rigid surjection.
  Analogously, since $g_i : (D, \Boxed{\sqsubset^D}) \to (B, \Boxed\sqsubset)$ is a rigid surjection, so is
  $\phi_i : (C \union D, \Boxed{\Boxed{\sqsubset^D} \oplus \Boxed{<^C}}) \to (B, \Boxed\sqsubset)$.
  Therefore, $\calE \in \Ob(\DD)$ and $(\phi_i, \phi_i) \in \hom_\DD(\calB, \calE)$ for all~$i$.
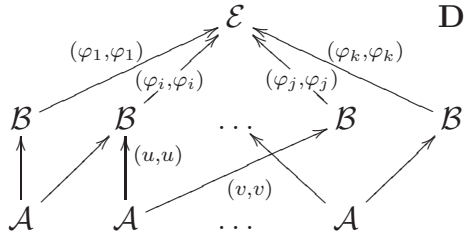
  $$
  \xymatrix{
    & & \calE & & \DD
  \\
    \calB \ar[urr]^{(\phi_1, \phi_1)} & \calB \ar[ur]|(0.4){(\phi_i, \phi_i)} & \ldots & \calB \ar[ul]|(0.4){(\phi_j, \phi_j)} & \calB \ar[ull]_{(\phi_k, \phi_k)}
  \\
    \calA \ar[u] \ar[ur] & \calA \ar[urr]_(0.5){(v,v)} \ar[u]_(0.65){(u,u)} & \ldots & \calA \ar[ur] \ar[ul] 
  }
  $$
  It is very easy to check that $(\phi_i, \phi_i) \cdot (u, u) = (\phi_j, \phi_j) \cdot (v, v)$ whenever
  $e_i \cdot (u, u) = e_j \cdot (v, v)$. Assume that $e_i \cdot (u, u) = e_j \cdot (v, v)$. Then
  $u \circ f_i = v \circ f_j$ and $u \circ g_i = v \circ g_j$
  (because $f \mathbin{\cdot_{\CC^\op}} u = u \mathbin{\cdot_{\CC}} f = u \circ f$ in this case).
  Now, take any $x \in C \union D$.
  If $x \in C$ then $u \circ \phi_i(x) = u \circ f_i(x) = v \circ f_j(x) = v \circ \phi_j(x)$.
  If, on the other hand, $x \in D$ then $u \circ \phi_i(x) = u \circ g_i(x) = v \circ g_j(x) = v \circ \phi_j(x)$.
  This concludes the proof.
\end{proof}

\section*{Acknowledgements}

The author would like to thank Erkko Lehtonen for a fruitful discussion which inspired this paper
and Miodrag Soki\'c for many insightful remarks.

\end{document}